\documentclass[10pt]{amsart}
\usepackage[english]{babel}
\usepackage{amssymb}
\usepackage{enumitem}
\usepackage{hyperref}
\usepackage[initials]{amsrefs}
\usepackage[all]{xy}
\SelectTips{cm}{}

\newcommand{\id}{\operatorname{id}}

\newcommand{\Stab}{\operatorname{Stab}}
\newcommand{\SL}{\operatorname{SL}}
\newcommand{\Fix}{\operatorname{Fix}}

\newcommand{\Aut}{\operatorname{Aut}}

\newcommand{\GL}{\operatorname{GL}}

\newcommand{\inn}{inn}

\newtheorem{theorem}{Theorem}[section]
\newtheorem{lemma}[theorem]{Lemma}

\newtheorem{prop}[theorem]{Proposition}

\theoremstyle{definition}

\newtheorem{defn}[theorem]{Definition}

\newtheorem{example}[theorem]{Example}

\newtheorem{assumption}[theorem]{Assumption}

\numberwithin{equation}{section}

\begin{document}

\title{Margulis' Super-Rigidity Theorem for non-lattices}

\author{Uri Bader}

\author{Alex Furman}



\maketitle

\begin{abstract}
We give a Super-Rigidity theorem a la Margulis which applies for a wider class of groups.
In particular it applies to subgroups which are not assumed to be lattices in the ambient group.
Our proof is based on the notion of Algebraic Representation of Ergodic Actions.
\end{abstract}

\section{Introduction}

Margulis' Super-Rigidity Theorem \cite[Theorem~VII.5.6]{margulis-book} deals with the setting of a lattice $\Gamma$ in a semisimple group $S$ and asserts that,
under certain assumptions, representations of $\Gamma$ could be continuously extended to $S$.
In this note we extend the scope of this theorem.
Let us make the following general definition.

\begin{defn}
Let $S$ be a locally compact group and $\Gamma<S$ a closed subgroup.
We will say that the inclusion $\Gamma<S$ is SR if
for every group $G$ which is the $k$-points of a connected adjoint $k$-simple algebraic group defined over a local field $k$
and every continuous homomorphism $\delta:\Gamma \to G$ 
which is unbounded and Zariski dense in $G$
there exists a unique continuous homomorphism $\rho:S\to G$
such that $\delta=\rho|_{\Gamma}$.
\end{defn}

\begin{theorem} \label{thm:main}
Let $S$ be a locally compact second countable group and $\Gamma<S$ a closed subgroup.
Assume that there exists a topologically generating (finite or infinite) sequence of abelian subgroup $T_0,T_1,\ldots <S$ such that for every $i>0$,
$T_i$ commutes with $T_{i-1}$ and for every $i\geq 0$, $T_i$ acts ergodically on $S/\Gamma$.
Assume also that $S/N$ has trivial abelianization, where $N$ is the closed normal subgroup generated by $\Gamma$.
Then $\Gamma<S$ is SR.
\end{theorem}

The proof of Theorem~\ref{thm:main} will be given in \S\ref{sec:main}.
We note that if $S$ is a connected, simply connected, semisimple group over a local field then it is generated by its root subgroup and these act ergodically on $S/\Gamma$ for every lattice $\Gamma<S$.
further, under a higher rank assumption on $S$ the root subgroups could be ordered in a sequentially commuting fashion.
Thus, our theorem above generalizes Margulis' Super-Rigidity Theorem.

\begin{example}
The inclusion $S=\SL_n(\mathbb{R})\ltimes \mathbb{R}^n$ and $\Gamma=\SL_n(\mathbb{Z})\ltimes \mathbb{Z}^n$ satisfies the hypothesis of 
Theorem~\ref{thm:main}.
For the groups $T_i$ we take the following sequentially commuting sequence of six abelian groups
\[
\begin{bmatrix}
    1 & * & \\
       & 1 & \\
       &    & 1
\end{bmatrix}
\begin{bmatrix}
    * \\
    0 \\
    0
\end{bmatrix},
\quad
\begin{bmatrix}
    1 & & * \\
       & 1 & \\
       &    & 1
\end{bmatrix}
\begin{bmatrix}
    0 \\
    0 \\
    0
\end{bmatrix},
\quad
\begin{bmatrix}
    1 & &   \\
       & 1 & *\\
       &    & 1
\end{bmatrix}
\begin{bmatrix}
    0 \\
    * \\
    0
\end{bmatrix},
\]
\[
\begin{bmatrix}
    1 & &  \\
    *  & 1 & \\
       &    & 1
\end{bmatrix}
\begin{bmatrix}
    0 \\
    0 \\
    0
\end{bmatrix},
\quad
\begin{bmatrix}
    1 & &   \\
       & 1 & \\
    *  &    & 1
\end{bmatrix}
\begin{bmatrix}
    0 \\
    0 \\
    *
\end{bmatrix},
\quad
\begin{bmatrix}
    1 & &   \\
       & 1 & \\
       & * & 1
\end{bmatrix}
\begin{bmatrix}
    0 \\
    0 \\
    0
\end{bmatrix}.
\]
To see that each of these groups acts ergodically on $S/\Gamma$ it is enough, by Howe-Moore Theorem, to see that $\SL_n(\mathbb{R})$
acts ergodically on $S/\Gamma$, or equivalently, that $\Gamma$ acts ergodically on $\mathbb{R}^n\simeq S/\SL_n(\mathbb{R})$.
The latter is equivalent to the ergodicty of $\SL_n(\mathbb{Z})$ on $\mathbb{R}^n/\mathbb{Z}^n$ which is well known.
\end{example}

\section{Proper representations} \label{sec:properreps}

In this section
we fix an lcsc group $T$ and a Lebesgue space $X$ on which $T$ acts ergodically.
We also fix a Polish group $G$ and a continuous homomorphism $\theta:T\to G$.
We will discuss the category of proper representations of $X$ associated with the above data.

By a proper action of $G$ we regard a Polish space $V$ endowed with an action of $G$
such that the map
\[ G\times V \to V\times V, \quad (g,v) \mapsto (v,gv) \]
is continuous and proper.
The properness assumption is equivalent to the assumption that for every precompact subsets $C_1,C_2\subset V$, the set $\{g~|~gC_1\cap C_2\neq\emptyset\}$
is precompact in $G$.

\begin{prop} \label{closed orbits}
If $V$ is a proper $G$ space than the $G$ orbits in $V$ are closed and the stabilizers are compact.
Furthermore, the orbit space $V/G$ is Hausdorff and
for every $v\in V$ the orbit map $G/G_v\to Gv$ is a homeomorphism.
\end{prop}

\begin{proof}
The fact that the stabilizers are compact follows at once from the definitions (taking $C_1=C_2$, a singleton),
noting that stabilizers are closed.
All remaining parts of the proposition follow from the fact that $V/G$ is Hausdorff
(the last statement follows from Effros' Lemma \cite[Lemma 2.5]{effros}).
This is a standard fact, which we now briefly recall:
any proper map into a metrizable space is closed, thus the set
$V\times V-\mbox{Im}(G\times V)$ is open in $V\times V$
and its image under the open map to $V/G\times V/G$ is again open
- thus its complement, i.e the diagonal, in the latter is closed.
\end{proof}

By a proper representation of $X$ we mean 
a proper action of $G$ on a Polish topological space $V$ 
together with a $T$-equivariant a.e defined measurable map $\phi:X\to V$.
By a standard abuse of notation we call such a proper representation $V$, and regard $\phi$ as $\phi_V$ in case of a possible confusion.
A morphism from a proper representation $V$ to a proper representation $U$ is given by a $G$-equivbariant continuous map $\alpha:V\to U$
such that $\phi_U$ agrees a.e with $\alpha\circ \phi_V$.
A proper representation is said to be transitive if the underlying action of $G$ on $V$ is transitive.

Note that if $X$ is a singleton and the image of $\Gamma$ in $G$ is non-compact then the category of proper representation of the $T$-space $X$ is empty.
This section is devoted to the study of proper representations, thus we make the following assumption.

\begin{assumption} \label{assum}
We will assume through out this section that the category of proper representation of the $T$-space $X$ is not empty.
\end{assumption}

\begin{theorem} \label{thm:initial proper}
There exists an initial object in the category of proper representations of the $T$-space $X$ and this initial object is transitive.
\end{theorem}

Before proving the theorem we will give some lemmas.

\begin{lemma} \label{lem:transunique}
Let $\phi_U:X\to U$ be a transitive proper representation of $X$
and 
let $\phi_V:X\to V$ be any proper representation of $X$.
A morphism of proper representations $U\to V$, if exists, is unique.
\end{lemma}

\begin{proof}
For two $G$-maps $\alpha_1,\alpha_2:U\to V$, the equalizer $\{u\in U\mid \alpha_1(u)=\alpha_2(u)\}$ is $G$-invariant,
thus empty or full, as $U$ is transitive.
\end{proof}

\begin{lemma} \label{lem:transproper}
Let $\phi_V:X\to V$ be a proper representation of $X$.
Then there exists a transitive proper representation $\phi_U:X\to U$
and a morphism of proper representations $\alpha:U\to V$.
\end{lemma}

\begin{proof}
Observer that the map $X\to V\to V/G$ obtained by composing $\phi_V$ with the obvious quotient map is $T$ invariant,
hence essentially constant, by ergodicity.
It follows that the image of $X$ is supported on a single $G$-orbit $U\subset V$.
\end{proof}

\begin{lemma} \label{lem:countableproducts}
Both the category of proper representations and the subcategory of transitive proper representations admit countable products.
\end{lemma}

\begin{proof}
Given a countable set of proper representations $\phi_i:X\to V_i$
it is clear that the the diagonal map $\phi:X\to \prod V_i$ forms their product in the category of proper representations.
Assuming the $V_i$'s where transitive, we use Lemma~\ref{lem:transproper}
to find a morphism $\alpha:U\to \prod V_i$ for some transitive proper representation $\phi_U:X\to U$
and check that the latter is their product in the subcategory of transitive proper representations.
\end{proof}

\begin{lemma} \label{lem:cofinal}
Let $K$ be a second countable compact group.
Then every chain of transitive $K$-spaces has a countable cofinal subchain.
More precisely,
let $A$ be a totally ordered poset, for every $\alpha\in A$ let $V_\alpha$ be a transitive $K$-space and for $\alpha\lneq \beta$ let $\theta_{\beta,\alpha}:V_\beta\to V_\alpha$ be a $K$-equivariant map which is not an isomorphism such that for $\alpha<\beta<\gamma$, $\theta_{\gamma,\alpha}=\theta_{\beta,\alpha}\circ\theta_{\gamma,\beta}$,
then there exists a countable cofinal subset in $A$. 
\end{lemma}

\begin{proof}
We assume as we may that $A$ has no greatest element.
We let $V_0$ be the inverse limit of the corresponding diagram in the category of compact $K$-spaces
and for each $\alpha\in A$ let $\theta_\alpha:V_0\to V_\alpha$ be the corresponding map.
We fix a base point $v\in V_0$ and for each $\alpha$ we denote the stabilizer of $\theta_\alpha(v)$ by $K_\alpha$.
As $K-\cap_A K_\alpha$ is second countable, its open cover $\{K-K_\alpha\}_{\alpha\in A}$ admits a countable subcover.
We obtain a countable subset $B$ in $A$ such that $\cap_A K_\alpha=\cap_B K_\beta$.
$B$ is cofinal in $A$, otherwise there exits $\alpha_0\in A$ such that for every $\beta\in B$, $\beta<\alpha_0$ and then $K_{\alpha_0} <\cap_B K_\beta=\cap_A K_\alpha$, but then for every $\alpha>\alpha_0$, $V_\alpha\to V_{\alpha_0}$ is an isomorphism and we conclude that $\alpha_0$ is a greatest element in $A$,
contradicting our assumption. 
\end{proof}

\begin{proof}[Proof of Theorem~\ref{thm:initial proper}]

In view of Lemma~\ref{lem:transproper} and Lemma~\ref{lem:transunique} it is enough to show that the subcategory of transitive proper representations has an
initial object.
By Lemma~\ref{lem:transunique} the set of isomorphism classes in this subcategory has a natural poset structure, one object is smaller than another if there exists a morphism from it to the other, and we need to show that this poset admits a least element.
By Lemma~\ref{lem:countableproducts} this poset is closed under taking countable minimums, 
thus it is enough to establish the existence of a minimal element
or, by Zorn's Lemma, a lower bound to every chain.
Having countable minimums, we are left to show that every chain has a countable cofinal subchain.
This is what we proceed to do.

Let $A$ be a linearly ordered poset and for every $\alpha\in A$ let $\phi_\alpha:X\to U_\alpha$ be a transitive proper representation of $X$
such that for $\alpha\lneq \beta$ there exists a morphism $U_\alpha\to U_\beta$ which is not an isomorphism.
Without loss of the generality we assume that $A$ has a greatest element $1$.
We fix a basepoint $u\in U_1$ and denote its stabilizer $K$.
Thus $K<G$ is a compact second countable subgroup.
For each $\alpha$ we denote by $\pi_\alpha:U_\alpha\to U_1$ the corresponding $G$-map and let $V_\alpha=\pi_\alpha^{-1}(\{u\})$.
The obvious chain of $K$-spaces thus obtained has a countable cofinal subchain by Lemma~\ref{lem:cofinal}
and we conclude that the chain $U_\alpha$ has a countable cofinal subchain as well.
\end{proof}

\begin{prop}[{cf. \cite[Theorem~4.7]{BF-products}}] \label{prop:yonedacomp}
Let $\phi:X\to G/K$ be an initial object in the category of proper representations of the $T$-space $X$.
Assume $T'$ is a locally compact group acting on $X$ commuting with $T$.
Then $\phi$ is also $T'$-equivariant with respect to a continuous homomorphism $\theta':T'\to N_{G}(K)/K$,
where $N_{G}(K)/K$ acts on $G/K$ on the right as usual.
\end{prop}

\begin{proof}
For a given $t'\in T'$ we consider the diagram
\begin{equation} \label{a-diag-AG}
\xymatrix{ X \ar[r]^{\phi} \ar[d]^{t'} & G/K \ar@{.>}[d]^{\theta'(t')}  \\
		 X \ar[r]^{\phi}   & G/K  }
\end{equation}
where we denote by $\theta'(t')$ the dashed arrow, which is the $G$-equivariant map $G/K\to G/K$ given by the fact that $\phi$ forms an initial object and
$\phi\circ t'$ is a proper representation of $X$.
By the uniqueness of the dashed arrow, the correspondence $t'\mapsto \theta'(t')$ is easily checked to form a homomorphism from $T'$ to the group of 
$G$-automorphisms of $G/K$ which we identify with $N/K$, where $N=N_{G}(K)$.
We are left to check the continuity of $\theta'$.

We let 
$U=L^0(X,G/K)$ and consider it with the Polish topology given by convergence in measure.
We endow $U$ with the action of $N/K$ by post-composition, and the action of $T'$  by precomposition.
By the fact that $N/K$ acts freely on $G/K$, we get that $N/K$ acts freely on $U$ as well.
Moreover, the $N/K$ orbits in $G/K$ are closed (as these are the fibers of $G/K\to G/N$),
thus also the $N/K$ orbits in $U$ are closed, as convergence in measure of a sequence implies a.e convergence of a subsequence.
It follows that the $N/K$-orbit map $N/K\to U$, $n\mapsto n\circ \phi$ is a homeomorphism onto its image, $N/K\cdot\phi$.
We let $\alpha:N/K\cdot\phi\to N/K$ be its inverse.

By the fact that the map $T'\times X \to G/K$, $(t',x)\mapsto \phi(t'x)$ is a.e defined and measurable, we get that the associated map $\beta:T'\to U$, $t'\mapsto \phi\circ t'$ is a.e defined and measurable, see \cite[Chapter VII, Lemma 1.3]{margulis-book}.
Since by the definition of $\theta'$, $\phi\circ t'=\theta'(t')\circ \phi$, we conclude that $\theta'$ agrees a.e with $\alpha\circ \beta$ which is a.e defined and measurable.
It follows that $\theta'$ is measurable.
By \cite[Lemma 2.1]{Rosendal} we conclude that $\rho'$ is a continuous homomorphism.
\end{proof}

\section{Only trivial proper representations}

In the previous section we fixed, apart of an lcsc group $T$ acting on a Lebesgue space $X$, 
also a Polish group $G$ and a continuous homomorphism $\theta:T\to G$, and studied the associated proper representations of $X$.
In this section we will vary $G$ and $\theta$.

\begin{defn} \label{def:trivprop}
Let $T$ be an lcsc group and $X$ a $T$-Lebesgue space.
We will say that the $T$-space $X$ has only trivial proper representations if 
for every Polish group $G$ and a continuous homomorphism $\theta:T\to G$,
for every proper $G$ space $V$ and associated proper representation $\phi:X\to V$,
$\phi$ is essentially constant.
\end{defn}

Our goal in this section is to prove the following theorem.

\begin{theorem} \label{thm:noproper}
Let $S$ be a locally compact second countable group and $\Gamma<S$ a closed subgroup.
Assume that there exists a topologically generating (finite or infinite) sequence of abelian subgroup $T_0,T_1,\ldots <S$ such that for every $i>0$,
$T_i$ commutes with $T_{i-1}$ and for every $i\geq 0$, $T_i$ acts ergodically on $S/\Gamma$.
Assume also that $S/N$ has trivial abelianization, where $N$ is the closed normal subgroup generated by $\Gamma$.
Then for every $i\geq 0$, the $T_i$-space $S/\Gamma$ has only trivial proper representations.
\end{theorem}

Before proving the theorem we will isolate out a major step in the proof that will be also used elsewhere.

\begin{lemma} \label{lem:extension}
Let $S$ be a locally compact second countable group and $\Gamma<S$ a closed subgroup.
Let $G$ be a Polish group and $\delta:\Gamma\to G$ a continuous homomoprhism.
Assume given a topologically generating (finite or infinite) sequence of subgroup $T_0,T_1,\ldots <S$,
for every $i\geq 0$ a continuous homomorphism $\theta_i:T_i\to G$
and an a.e defined measurable map $\phi:S\to G$ which is left $\Gamma$ equivariant and right $T_i$-equivariant for each $i$.
That is, for every $i$, for Haar a.e $s\in S$, for every $\gamma\in \Gamma$ and every $t\in T_i$ we have
\[ \phi(\gamma s t)=\delta(\gamma)\phi(s)\theta_i(t). \]
Then there exists a continuous homomorphism $\theta:S\to G$ such that for every $i$, $\theta_i=\theta|_{T_i}$ and under which $\phi$ is right $S$-equivariant.
Furthermore, there exists $g\in G$ such that $\delta=\inn(g)\circ \theta|_\Gamma$, where $\inn(g):G\to G$ denotes the map $x\mapsto gxg^{-1}$.
In particular, $\delta$ extends to $S$.
\end{lemma}

\begin{proof}
Using \cite[Lemma~5.1]{BF-MSR} with $X=S$ we get a continuous homomorphism $\theta:S\to G$ such that for every $i$, $\theta_i=\theta|_{T_i}$
and $\phi:S\to G$ is right $S$-equivariant with respect to $\theta:S\to G$.
It follows that the a.e defined measurable map $\Phi:S\to G$, $\Phi(s)=\phi(s)\theta(s)^{-1}$ is right $S$-invariant, thus essentially constant.
Denoting the essential value of $\Phi$ by $g$, we get that $\phi(s)=g\theta(s)$ a.e.
As $\phi$ is left $\Gamma$-equivariant, we get $g\theta(\gamma s)=\delta(\gamma)g\theta(s)$ for a.e $s$.
Fixing $\gamma$, this equation of continuous functions in the parameter $s$ is satisfied a.e on $S$, hence everywhere.
Taking $s=e$ we get that indeed $\delta=\inn(g)\circ \theta|_\Gamma$.
\end{proof}

\begin{proof}[Proof of Theorem~\ref{thm:noproper}]
We will prove the theorem for $i=0$. The general case follows by reordering the groups $T_i$.
We denote $X=S/\Gamma$.
We fix a Polish group $G$, a homomorphism $\theta_0:T_0\to G$,
a proper $G$ space $V$ and an associated proper representation $\phi:X\to V$.
We argue to show that $\phi$ is essentially constant.
Replacing $G$ with the closure of $\rho(T_0)$, we assume as we may that $\rho(T_0)$ is dense in $G$.
In particular, $G$ is abelian.
We also assume as we may that $\phi$ is an initial object in the category of proper representations of the $T_0$ -space $S/\Gamma$.
In particular, $V$ is $G$ transitive. 
Moding up the kernel of the action of $G$ on $V$, we also assume as we may that this action is faithful.
We thus may identify $V$ with $G$ itself, endowed with the left regular action.

Setting $T'=T_1$ in Proposition~\ref{prop:yonedacomp} and noting that $K=\{e\}$ and $N_G(K)=G$,
we get that $\phi$ is $T'$-equivariant with respect to a continuous homomorphism $\theta_1:T_1\to G$.
We thus may regard $\phi:X\to G$ as a proper representation of $X$ as a $T_1$-space.
In particular, Assumption~\ref{assum} is satisfied for the category of proper representations of $X$ as a $T_1$-space,
so by Theorem~\ref{thm:initial proper} this category has an initial object which is transitive.
We now claim that $\phi$ is actually an initial object also in the category of proper representations of $X$ as a $T_1$-space.
Let $\phi_1:X\to V_1$ be such an initial object.
As $\phi:X\to G$ is another object in the category, there exists a $G$-map $\alpha:V_1\to G$ representing a morphism from $\phi_1$ to $\phi$.
We conclude that the action of $G$ on $V_1$ is faithful, thus $V_1$ is a simply transitive $G$-space.
It follows that the map $\alpha$ is actually an isomorphism, thus indeed $\phi:X\to G$ is an initial object, 
being isomorphic to an initial object.
Proceeding by induction in an obvious way, we get for every $i\geq 0$ a continuous homomorphism $\theta_i:T_i\to G$
with respect to which the map $\phi$ is $T_i$-equivariant.

We view $\phi$ as a left $\Gamma$-invariant function $S\to G$ which is right $T_i$-equivariant via $\theta_i$ for each $i$.
Setting $\delta:\Gamma\to G$ to be the trivial homomorphism,
we get by Lemma~\ref{lem:extension} that 
there exists a continuous homomorphism $\theta:S\to G$ such that for every $i$, $\theta_i=\theta|_{T_i}$ and $g\in G$ such that $\delta=\inn(g)\circ \theta|_\Gamma$.
From the last condition we get that $\Gamma$ is in the kernel of $\theta$, hence so is also the closed normal subgroup generated by $\Gamma$, $N$.
As $S/N$ has trivial abelianization, 
we conclude that $\theta$ is trivial.
In particular we get that $\theta_0=\theta|_{T_0}$ is trivial,
which concludes the proof, as $\theta_0(T_0)$ is dense in $G$ and $V=G$.
\end{proof}

\section{$T$-proper representations of $S$}

Throughout this section we fix a locally compact second countable group $S$ and a closed subgroup $\Gamma<S$.
We endow $S$ with its Haar measure and regard it as a Lebesgue space.
We also fix 
a Polish group $G$ and a continuous homomorphism $\delta:\Gamma\to G$.

\begin{defn}
Given all the data above,
for a closed subgroup $T<S$ a $T$-proper representation of $S$
consists of the following data:
\begin{itemize}
\item a Polish space $V$ endowed with a proper action of $G$.
\item a Polis group $L$ acting faithfully on $V$ commuting with the $G$-action.
This action will be considered as a right action, while the $G$-action on $V$ will be considered on the left.
\item a continuous homomorphism with dense image $\theta:T\to L$.
\item an associated representation of the $\Gamma\times T$-space $S$ on $V$,
where $\Gamma$ acts on the left, $T$ acts on the right of the Lebesgue space $S$.
That is, a Haar a.e defined measurable map $\phi:S \to V$ such that for almost every $s\in S$, every $\gamma\in \Gamma$ and every $t\in T$,
\[ \phi(\gamma s t)=\delta(\gamma) \phi(s)d(t). \]
\end{itemize}
We abbreviate the notation by saying that $V$ is a $T$-proper representation of $S$,
denoting the extra data by $L_V,\theta_V$ and $\phi_V$.
Given another $T$-proper representation $U$ ,
a morphism of $T$-proper representations of $S$ from the $T$-proper representation $U$ to the $T$-proper representation $V$ is
a continuous map $\alpha:U\to V$ which is $G \times T$-equivariant,
and such that $\phi_V$ agrees a.e with $\alpha\circ \phi_U$.
\end{defn}

We say that $V$ is a transitive $T$-proper representation if the underlying action of $G$ on $V$ is transitive.
In that case we may identify $V$ with $G/K$ for some compact subgroup $K<G$ and $L$ with a closed subgroup of $N_G(K)/K$ which acts on $G/K$ as usual,
on the right. Note that the action of $N_G(K)/K$ on $G/K$ is proper.

Similarly to making Assumption~\ref{assum} in \S\ref{sec:properreps} we make the following assumption.

\begin{assumption} \label{Tassum}
We will assume through out this section that the category of $T$-proper representations of $S$ is not empty.
\end{assumption}

\begin{theorem} \label{thm:Tinitial proper}
Let $T<S$ be a closed abelian subgroup such that the $T$-space $S/\Gamma$ has only trivial proper representations.
Then the category of $T$-proper representations of $S$ has an initial object
and this initial object is a transitive. 
\end{theorem}

The proof is, mutatis mutandis, the same as the proof of Theorem~\ref{thm:initial proper}.
We will only need to replace Lemmas~\ref{lem:transunique}, \ref{lem:transproper} and \ref{lem:countableproducts}
by the following analogues.
The first one is no more than a reformulation
and its proof is literally the same as the proof of Lemma~\ref{lem:transunique}.

\begin{lemma} \label{Tlem:transunique}
Let $T<S$ be a closed subgroup.
Let $\phi_U:S\to U$ be a transitive $T$-proper representations of $S$
and 
let $\phi_V:S\to V$ be any $T$-proper representations of $S$.
A morphism of proper representations $U\to V$, if exists, is unique.
\end{lemma}

For the next lemma we actually use our assumptions on $T$ and $\Gamma$.

\begin{lemma} \label{lem:Ttransproper}
Let $T<S$ be a closed abelian subgroup such that the $T$-space $S/\Gamma$ has only trivial proper representations.
Let $\phi_V:S\to V$ be a $T$-proper representations of $S$.
Then there exists a transitive $T$-proper representations of $S$, $\phi_U:X\to U$,
and a morphism of $T$-proper representations of $S$, $\alpha:U\to V$.
\end{lemma}

\begin{proof}
We first observe that the $\Gamma\times T$-action on $S$ is ergodic.
Indeed, otherwise we could find a non-constant invariant map $S\to \{0,1\}$ which could be seen as a proper representation of the $T$-space $S/\Gamma$ with respect to the trivial homomorphism $\theta:T\to \{e\}$ and the trivial action of $\{e\}$ on $V=\{0,1\}$, contradicting the assumption that 
the $T$-space $S/\Gamma$ has only trivial proper representations. 
Considering $X=S$ as an ergodic $\Gamma\times T$-space and the continuous homomorphism $\delta\times \theta_V:\Gamma\times T\to G\times L_V$, we consider $\phi_V:S\to V$ as proper representation of $X$
and conclude by Lemma~\ref{lem:transproper}
the existence of a transitive proper $G\times L_V$-space $U$, a proper representation of the $\Gamma\times T$-space $X$, $\phi_U:X\to U$,
and a $G\times L_V$-map $\alpha:U\to V$.

We next claim that $U$ is in fact $G$-transitive,
alternatively, that $U/G$ is trivial.
Composing $\phi_U$ with $U\to U/G$ we obtain a map $S\to U/G$ which is $\Gamma$-invariant, hence factors via $S/\Gamma$ giving a map
$S/\Gamma\to U/G$. 
Notice that $U/G$ is a transitive $L_V$-space and that $L_V$ is abelain, as it contains the dense abelain subgroup $\theta_V(T)$.
We let $L=L_V/N$, where $N$ is the kernel of the action of $L_V$ on $U/G$ and let $\delta:T\to L$ be the composition of $\theta_V$ with the quotient map $L_V \to L$. We now regard the map $\psi:S/\Gamma\to U/G\simeq L$ as a proper representation of the $T$-space $S/\Gamma$ associated with the continuous homomorphism $\delta:T\to L$.
As the $T$-space $S/\Gamma$ has only trivial proper representations, $\psi$ is essentially constant. Its essential image is clearly $\delta(T)$-invariant,
thus also $L$-invariant as $\delta(T)$ is dense in $L$ ($\theta_V$ is dense in $L_V$ and $L_V\to L$ is surjective).
It follows that $L$ is the trivial group and thus $U/G\simeq L$ is trivial too.
This proves the claim.

Finally, 
we let $L_U=L_V/K$, where $K$ is the kernel of the action of $L_V$ on $U$ and let $\theta_U:T\to L_U$ be the composition of $\theta_V$ with the quotient map $L_V \to L_U$. 
By doing so, the map $\phi_U:S\to U$ becomes a transitive $T$-proper representations of $S$
and $\alpha:U\to V$ becomes the required morphism of $T$-proper representations of $S$.
\end{proof}

\begin{lemma} \label{lem:Tcountableproducts}
Let $T<S$ be a closed abelian subgroup such that the $T$-space $S/\Gamma$ has only trivial proper representations.
Then both the category of $T$-proper representations of $S$ and the subcategory of transitive $T$-proper representations of $S$ admit countable products.
\end{lemma}

\begin{proof}
Assume having a countable set of $T$-proper representations of $S$, given by the data represented by $\theta_i:T\to L_i$ and $\phi_i:S\to V_i$.
We let $L<\prod_i L_i$ be the closure of the image of $T$ under the diagonal map $\theta:T\to \prod_i L_i$ and note that it acts freely on $V=\prod_i V_i$.
It is now straight forward to check that the diagonal map $\phi:S\to V$ together with $\theta:T\to L$ form the product of the given countable 
set of $T$-proper representations of $S$.

Assuming the $V_i$'s where transitive, we use Lemma~\ref{lem:Ttransproper}
to find a morphism $\alpha:U\to \prod V_i$ for some transitive $T$-proper representations of $S$
and check that the latter is their product in the subcategory of transitive $T$-proper representations of $S$.
\end{proof}

It turns out that an initial object in the category of $T$-proper representations of $S$ extends naturally to an $N$-algebraic representation of $S$, where $N$ denotes the normalizer of $T$ in $S$.

\begin{prop} \label{prop:Tyonedacomp}
Let $T<S$ be a closed abelian subgroup such that the $T$-space $S/\Gamma$ has only trivial proper representations
and let $G/K$,
$\theta:T \to L<N_G(K)/K$ and $\phi:S \to G/K$ be an initial object
in the category of $T$-proper representations of $S$, as guaranteed by Theorem~\ref{thm:Tinitial proper}.
Then the map $\theta:T \to N_G(K)/K$ extends to the normalizer of $T$ in $S$, $N=N_S(T)$,
and the map $\phi$ could be seen as an $N$-proper representations of $S$.
More precisely, there exists a continuous homomorphism $\bar{\theta}:N\to N_G(K)/K$ satisfying $\bar{\theta}|_T=\theta$
such that, denoting by $\bar{L}$ the closure of $\bar{\theta}$ in $N_G(K)/K$, the data
$G/K$,
$\bar{\theta}:N \to \bar{L}<N_G(K)/K$ and $\phi:S \to G/K$ forms an $N$-proper representation.

In particular, if $T'<S$ is a closed group commuting with $T$ then $\theta'=\bar{\theta}|_{T'}:T' \to N_G(K)/K$
is a continuous homomorphism 
with respect to which $\phi$ is $T'$-equivariant.
\end{prop}

The proof is essentially the same as the proof of \cite[Theorem~4.6]{BF-MSR}.

\begin{proof}
Fix $n\in N$.
Set $d'=d\circ \inn(n):T\to L$, where $\inn(n):T\to T$ denotes the inner automorphism $t\mapsto ntn^{-1}$,
and $\phi'=\phi\circ \rho(n):S\to G/K$, where $\rho(n):S\to S$ denotes the right regular action $s\mapsto sn^{-1}$.
We claim that the data $L$, $G/K$, $d'$ and $\phi'$ forms a new $T$-proper representation of $S$.
Indeed, for almost every $s\in S$, every $\gamma\in \Gamma$ and every $t\in T$,
\[ \phi'(\gamma s t)=\phi(\gamma s t n^{-1})=\phi(\gamma s n^{-1} t^n)= \delta(\gamma) \phi' (s)d'(t). \]

By the fact that the $T$-proper representation of $S$ given by $L$, $G/K$,
$d$ and $\phi$ forms an initial object we get the dashed vertical arrow, which we denote $\bar{d}(n)$, in the following diagram.
\begin{equation} \label{diag-AG}
\xymatrix{ S \ar[r]^{\phi} \ar[rd]_{\phi'} & G/K \ar@{.>}[d]^{\bar{d}(n)}  \\
		    & G/K  }
\end{equation}
By the uniqueness of the dashed arrow, the map $n\mapsto \bar{d}(n)$ is easily checked to form a homomorphism from $N$ to the group
of $G$-automorphism of $G/K$,
which we identify with $N_G(K)/K$.
Note that for $n\in T$, $d(t):G/K\to G/K$ could also be taken to be the dashed arrow, thus $\bar{d}(t)=d(t)$, by uniqueness.
Therefore the homomorphism $\bar{d}$ extends $d$.
The fact that such a homomorphism is necessarily continuous is explained in the proof of \cite[Theorem~4.7]{BF-products}.
We define $\bar{L}$ to be the closure the image of $\bar{d}$.
We thus indeed obtain an $N$-representation of $S$, given by the group $\bar{L}$, the coset space $G/K$, the homomorphism $\bar{d}:N \to \bar{L}$ and the (same old) map
$\phi:S \to G/K$.
\end{proof}

\section{Only trivial $T$-proper representations of $S$} \label{sec:trivproper}

In the previous section we fixed, apart of an lcsc group $S$ and its closed subgroup $\Gamma$ also  
also a Polish group $G$ and a continuous homomorphism $\delta:\Gamma\to G$, and studied the associated $T$-proper representations of $S$.
In this section we will vary $G$ and $\delta$.

\begin{example}
If either $\delta:\Gamma\to G$ extends to $S$ or it has a precompact image then Assumption~\ref{Tassum} is always satisfied.
To see the first case, assuming $d:S\to G$ is a continuous homomorphism such that $\delta=d|_\Gamma:\Gamma\to G$, 
setting $V=G$, $\phi=d:S\to G$
and for any given $T<S$ letting $L$ be the closure of $d(T)$ in $G$ and $\theta=d|_T:T\to L$,
we get a $T$-proper representations of $S$.
For the second case, assuming $K=\overline{\delta(\Gamma)}$ is compact
setting $V=G/K$, $\phi:S \to V$ be the constant map to the trivial coset
and for any given $T<S$ letting $\theta:T\to L=\{e\}$ be the trivial morphism,
we get a $T$-proper representations of $S$.
\end{example}

The above two examples are extreme cases of what we call trivial $T$-proper representations of $S$.

\begin{defn} \label{def:trivprop}
Fix an lcsc group $S$ and a closed subgroup $\Gamma<S$.
Let $T<S$ be a closed subgroup.
We will say that $S$ has only trivial $T$-proper representations if 
for any Polish group $G$ and a continuous homomorphism $\delta:\Gamma\to G$
such that there exits a corresponding $T$-proper representations of $S$,
there exit a compact subgroup $K<G$ normalized by $\delta(\Gamma)$,
a continuous homomorphism $\theta:S\to N_G(K)/K$ and $n\in N_G(K)$ such that the composition of $\inn(n)\circ\delta:\Gamma\to N_G(K)$ with the
quotient $N_G(K)\to N_G(K)/K$ equals $\theta|_\Gamma$
and there exists an a.e defined measurable $\phi:S\to G/K$ which is left $\Gamma$-equivariant via $\delta$ and right $S$-equivariant via $\theta$
such that $\phi$ together with $\theta|_T$
form an initial object in corresponding category of 
$T$-proper representations of $S$.
\end{defn}

Our goal in this section is to prove the following theorem.

\begin{theorem} \label{thm:Tnoproper}
Let $S$ be a locally compact second countable group and $\Gamma<S$ a closed subgroup.
Assume that there exists a topologically generating (finite or infinite) sequence of abelian subgroup $T_0,T_1,\ldots <S$ such that for every $i>0$,
$T_i$ commutes with $T_{i-1}$ and for every $i\geq 0$, $T_i$ acts ergodically on $S/\Gamma$.
Assume also that $S/N$ has trivial abelianization, where $N$ is the closed normal subgroup generated by $\Gamma$.
Let $T=T_0$.
Then $S$ has only trivial $T$-proper representations.
\end{theorem}

\begin{proof}
Let $G$ be a Polish group and $\delta:\Gamma\to G$ be a continuous homomorphism
such that there exits a corresponding $T$-proper representations of $S$.
By Theorem~\ref{thm:Tinitial proper} the category of $T$-proper representations of $S$
has a transitive initial object.
We let $\phi:S\to U$ and $\theta_0:T=T_0\to L_0$ be the data associated with this initial object.
We pick a base point in $u\in U$ which is in the support of the measure class pushed by $\phi$ from the Haar measure on $S$
and identify accordingly $U=G/K$ where $K$ is the compact group stabilizing $u$.
We let $N$ be the normalizer of $K$ in $G$ and also identify accordingly $L_0$ as a closed subgroup of $N/K$.
We will argue to show that $\delta(\Gamma)<N$
and there exists a continuous homomorphism $\theta:S\to N/K$ extending $\theta_0$ such that $\phi$ is right $S$-invariant via $\theta$
and there exists $n\in N$ such that the composition of $\inn(n)\circ\delta:\Gamma\to N$ with the
quotient $N\to N/K$ equals $\theta|_\Gamma$.
This will prove the theorem.

Setting $T'=T_1$ in Proposition~\ref{prop:Tyonedacomp}
we get a continuous homomorphism $\theta_1:T_1\to L_1<N/K$ with respect to which $\phi$ is also $T_1$-equivariant.
We thus may regard $\phi:S\to G/K$ as a $T_1$-proper representation of $S$.
In particular, Assumption~\ref{Tassum} is satisfied for the category of $T_1$-proper representation of $S$,
so by Theorem~\ref{thm:Tinitial proper} this category has also an initial object which is transitive.

We now claim that $\theta_1$ and $\phi$ form an initial object also in the category of $T_1$-proper representation of $S$.
Letting $\phi_1:S\to G/K_1$ and $\theta'_1:T_1\to L'_1$ be an initial object in the category of $T_1$-proper representations of $S$,
we will prove the claim by showing that this object is isomorphic in the category to the object formed by $\phi$ and $\theta_1$.
Clearly we have a unique morphism from $\phi_1$ to $\phi$, represented by a $G$-map $\alpha:G/K_1\to G/K$
such that $\phi$ agrees a.e with $\alpha\circ \phi_1$.
We need to show that $\alpha$ is invertible as a morphism in the category, which is the same as being invertible as a $G$-map.
Applying Proposition~\ref{prop:Tyonedacomp} again, interchanging the role of $T_0$ and $T_1$ we get a continuous homomorphism $\theta'_0:T_0\to L'_0<N_G(K_1)/K_1$ with respect to which $\phi_1$ is an object in the category of $T_0$-proper representation of $S$ in which $\theta$ and $\phi$ form an initial object.
We thus get a $G$-map $\beta:G/K_1\to G/K$
such that $\phi_1$ agrees a.e with $\beta\circ \phi$.
It follows that $\phi$ agrees a.e with $\beta\circ\alpha\circ \phi$ and in particular, $\beta\circ\alpha=\id_{G/K}$ as this two $G$-maps of a $G$-transitive space agree somewhere. Similarly, $\alpha\circ \beta=\id_{G/K_1}$.
It follows that $\alpha$ is indeed invertible and this proves the claim.

Proceeding by induction in an obvious way, we get for every $i\geq 0$ a continuous homomorphism $\theta_i:T_i\to N/K$
with respect to which the map $\phi$ is $T_i$-equivariant.
The space of orbits of the right action of $N/K$ on $G/K$ is naturally identified with $G/N$ and the standard map $G/K\to G/N$ is identified with the quotient map with respect to this action.
We compose now the map $\phi:S\to G/K$ with this map $G/K\to G/N$, getting a map $S\to G/N$ which is right $T_i$-invariant for every $i$.
As the groups $T_i$ topologically generate $S$, we get that $S\to G/N$ is right $S$-invariant, hence essentially constant.
It follows that the image of $\phi$ is supported on a unique $N/K$ orbit in $G/K$.
As the base point $u\in U\simeq G/K$ is in this support, we conclude that the image of $\phi$ is supported in the subspace $N/K\subset G/K$.
As this support is $\delta(\Gamma)$-invariant, it follows that $\delta(\Gamma)<N$.
By restricting the codomain we view $\delta$ as a homomorphism from $\Gamma$ to $N$ and
composing it with the quotient map $N\to N/K$ we obtain the homomorphism $\delta':\Gamma\to N/K$.
Restricting the codomain we also view $\phi$ as an a.e defined measurable function $S\to N/K$ which satisfies
for every $i$, for Haar a.e $s\in S$, for every $\gamma\in \Gamma$ and every $t\in T_i$,
\[ \phi(\gamma s t)=\delta'(\gamma)\phi(s)\theta_i(t). \]
Applying Lemma~\ref{lem:extension} in this context
we get that 
there exists a continuous homomorphism $\theta:S\to N/K$ such that for every $i$, $\theta_i=\theta|_{T_i}$ and under which $\phi$ is right $S$-equivariant.
Furthermore, there exists $nK\in N/K$ such that $\delta'=\inn(nK)\circ \theta|_\Gamma$.
In particular we have that $\theta_0=\theta|_{T_0}$
and for $n\in N$ the composition of $\inn(n)\circ\delta:\Gamma\to N$ with the
quotient $N\to N/K$ equals $\theta|_\Gamma$.
This concludes our proof.
\end{proof}


\section{Algebraic representation of ergodic actions} \label{sec:AREA}

In this section we fix
a field $k$ with a non-trivial absolute value which is separable and complete (as a metric space)
and a $k$-algebraic group ${\bf G}$.
We note that ${\bf G}(k)$ has the structure of a Polish topological group, see \cite[Proposition~2.2]{BDL}.
We also fix a locally compact second countable group $T$ and a continuous homomorphism $d:T\to {\bf G}(k)$,
where ${\bf G}(k)$ is considered with its Polish topology.
We let $X$ be a Lebesgue $T$-space.

\begin{defn}
Given all the data above, an algebraic representation of $X$
consists of a $k$-${\bf G}$-algebraic variety ${\bf V}$ and
an a.e defined measurable map $X:T \to {\bf V}(k)$ such that for almost every $x\in X$ and every $t\in T$,
\[ \phi(tx)=d(t)\phi(x). \]
Sometimes we abbreviate the notation by saying that ${\bf V}$ is an algebraic representation of $X$,
and denote $\phi$ by $\phi_{\bf V}$ for clarity.
A morphism from the algebraic representation ${\bf U}$ to the algebraic representation ${\bf V}$ consists of
a $k$-algebraic map $\psi:{\bf U}\to {\bf V}$ which is ${\bf G}$ equivariant,
and such that $\phi_{\bf V}$ agrees almost everywhere with $\psi\circ \phi_{\bf U}$.
An algebraic representation ${\bf V}$ of $X$ is said to be a coset algebraic representation
if in addition
${\bf V}$ is isomorphic as an algebraic representation to a coset variety ${\bf G}/{\bf H}$ for some $k$-algebraic subgroup ${\bf H}<{\bf G}$.
\end{defn}

The following is one of the basic observations of the theory of algebraic representations of ergodic actions.
For a proof and discussion we refer the reader to \cite{BF-products}.

\begin{prop}[{\cite[Proposition 4.2]{BF-products}}] \label{prop:AG-ergodic}
Assume $X$ is $T$-ergodic.
Then for every algebraic representation $\phi_{\bf V}:X\to {\bf V}(k)$ there exists a coset representation $\phi_{{\bf G}/{\bf H}}:X \to{\bf G}/{\bf H}(k)$
and a morphism of algebraic representations $i:{\bf G}/{\bf H}\to {\bf V}$,
that is $i$ is a $\bf G$-equivariant $k$-morphism such that for a.e $x\in X$, $\phi_{\bf V}(x)=i\circ \phi_{{\bf G}/{\bf H}}(x)$.
\end{prop}

\begin{defn} \label{def:trivalg}
We will say that the $T$-space $X$ has only trivial algebraic representations if 
every algebraic representation $\phi_{\bf V}:X\to {\bf V}(k)$ is essentially constant.
\end{defn}

\begin{prop}
If $T$ is abelian and the $T$-space $X$ has only trivial proper representations as in Definition~\ref{def:trivprop}
then the $T$-space $X$ has only trivial algebraic representations.
\end{prop}

\begin{proof}
Let $\phi_{\bf V}:X\to {\bf V}(k)$ be an algebraic representation of the $T$-space $X$.
By replacing ${\bf G}$ with the Zariski closure of $d(T)$ we may assume that ${\bf G}$ is abelian.
By Proposition~\ref{prop:AG-ergodic} we may assume that ${\bf V}={\bf G}/{\bf H}$ for some $k$-algebraic subgroup ${\bf H}<{\bf G}$.
By further replacing ${\bf G}$ with the quotient group ${\bf G}/{\bf H}$ we may assume that in fact ${\bf V}={\bf G}$.
We thus get that $\phi:X\to {\bf G}(k)$ is a proper representation of the $T$-space $X$ and conclude that it is essentially constant.
\end{proof}

\section{$T$-algebraic representations of $S$}

This section introduces a slight reformulation of \cite[\S4]{BF-MSR}.
The main differences are that $\Gamma<S$ below is not assumed to be a lattice here and we relay on Definition~\ref{def:trivalg} rather then on a weak mixing assumption.
Throughout this section we fix a locally compact second countable group $S$ and a closed subgroup $\Gamma<S$.
We endow $S$ with its Haar measure and regard it as a Lebesgue space.
We also fix a field $k$ endowed with a non-trivial absolute value which is separable and complete (as a metric space),
and a $k$-algebraic group $\bf{G}$.
We denote by $G$ the Polish group ${\bf{G}}(k)$.
Finally, we fix a continuous homomorphism $\delta:\Gamma\to G$.

\begin{defn}
Given all the data above,
for a closed subgroup $T<S$ a $T$-algebraic representation of $S$
consists of the following data:
\begin{itemize}
\item a $k$-algebraic group ${\bf L}$,
\item a $k$-$({\bf G}\times {\bf L})$-algebraic variety ${\bf V}$, regarded as a left $\bf G$, right $\bf L$ space,
on which the $\bf L$-action is faithful,
\item a continuous homomorphism $d:T\to {\bf L}(k)$ with a Zariski dense image,
\item an associated algebraic representation of the $\Gamma\times T$-space $S$ on $\bf V$,
where $\Gamma$ acts on the left, $T$ acts on the right of the Lebesgue space $S$.
That is, a Haar a.e defined measurable map $\phi:S \to {\bf V}(k)$ such that for almost every $s\in S$, every $\gamma\in \Gamma$ and every $t\in T$,
\[ \phi(\gamma s t)=\delta(\gamma) \phi(s)d(t). \]
\end{itemize}
We abbreviate the notation by saying that ${\bf V}$ is a $T$-algebraic representation of $S$,
denoting the extra data by ${\bf L}_{\bf V}, d_{\bf V}$ and $\phi_{\bf V}$.
Given another $T$-algebraic representation ${\bf U}$ we let ${\bf L}_{{\bf U},{\bf V}}<{\bf L}_{\bf U}\times {\bf L}_{\bf V}$ be the Zariski closure of the image of $d_{\bf U}\times d_{\bf V}:T\to {\bf L}_{\bf U}\times {\bf L}_{\bf V}$. ${\bf L}_{{\bf U},{\bf V}}$ acts on ${\bf U}$ and ${\bf V}$ via its projections to ${\bf L}_{\bf U}$ and ${\bf L}_{\bf V}$ correspondingly.
A morphism of $T$-algebraic representations of $S$ from the $T$-algebraic representation ${\bf U}$ to the $T$-algebraic representation ${\bf V}$ is
a $k$-algebraic map $\psi:{\bf U}\to {\bf V}$ which is ${\bf G}\times{\bf L}_{{\bf U},{\bf V}}$-equivariant,
and such that $\phi_{\bf V}$ agrees a.e with $\psi\circ \phi_{\bf U}$.
\end{defn}

Fix a $k$-subgroup ${\bf H}<{\bf G}$ and denote ${\bf N}=N_{\bf G}({\bf H})$.
This is again a $k$-subgroup.
Any element $n\in {\bf N}$ gives a ${\bf G}$-automorphism of ${\bf G}/{\bf H}$ by
$g{\bf H}\mapsto gn^{-1}{\bf H}$.
It is easy to see that the homomorphism ${\bf N} \to \Aut_{\bf G}({\bf G}/{\bf H})$ thus obtained is onto and its kernel is ${\bf H}$.
Under the obtained identification ${\bf N}/{\bf H} \simeq \Aut_{\bf G}({\bf G}/{\bf H})$,
the $k$-points of the $k$-group ${\bf N}/{\bf H}$ are identified with the $k$-${\bf G}$-automorphisms of ${\bf G}/{\bf H}$.

\begin{defn}
A $T$-algebraic representation of $S$ is said to be a coset $T$-algebraic representation if it is isomorphic as a $T$-algebraic representation to
${\bf G}/{\bf H}$ for some $k$-algebraic subgroup ${\bf H}<{\bf G}$,
and ${\bf L}$ corresponds to a $k$-subgroup of $N_{\bf G}({\bf H})/{\bf H}$ which acts on ${\bf G}/{\bf H}$ as described above.
\end{defn}

It is clear that the collection of $T$-algebraic representations of $S$ and their morphisms form a category.

\begin{theorem} \label{thm:initialproper}
Assume the $T$ space $S/\Gamma$ has only trivial algebraic representations.
Then the category of $T$-algebraic representations of $S$ has an initial object
and this initial object is a coset $T$-algebraic representation.
\end{theorem}

\begin{proof}
The proof is the same as the proof of \cite[Theorem~4.3]{BF-MSR}.
The only difference is that the assumption that $S/\Gamma$ has only trivial algebraic representations
is made their only implicitly, via \cite[Proposition~3.3]{BF-MSR} (which is used in \cite[Lemma~4.4]{BF-MSR}) using the weakly mixing assumption appearing in
the conditions of \cite[Theorem~4.3]{BF-MSR}.
\end{proof}

It turns out that an initial object in the category of $T$-algebraic representations of $S$ extends naturally to an $N$-algebraic representation of $S$, where $N$ denotes the normalizer of $T$ in $S$.

\begin{prop} \label{prop:algnormalizer}
Assume the action of $T$ on $S/\Gamma$ has only trivial algebraic representations and let ${\bf G}/{\bf H}$,
$\theta:T \to {\bf L}(k)<N_{\bf G}({\bf H})/{\bf H}(k)$ and $\phi:S \to {\bf G}/{\bf H}(k)$ be an initial object
in the category of $T$-algebraic representations of $S$, as guaranteed by Theorem~\ref{thm:initialproper}.
Then the map $\theta:T \to N_{\bf G}({\bf H})/{\bf H}(k)$ extends to the normalizer of $T$ in $S$, $N=N_S(T)$,
and the map $\phi$ could be seen as an $N$-algebraic representations of $S$.
More precisely, there exists a continuous homomorphism $\bar{\theta}:N\to N_{\bf G}({\bf H})/{\bf H}(k)$ satisfying $\bar{\theta}|_T=\theta$
such that, denoting by $\bar{\bf L}$ the Zariski closure of $\bar{\theta}$ in $N_{\bf G}({\bf H})/{\bf H}$, the data
${\bf G}/{\bf H}$,
$\bar{\theta}:N \to \bar{\bf L}(k)<N_{\bf G}({\bf H})/{\bf H}(k)$ and $\phi:S \to ({\bf G}/{\bf H})(k)$ forms an $N$-algebraic coset representation.

In particular, if $T'<S$ is a closed group commuting with $T$ then $\theta'=\bar{theta}|_{T'}:T' \to N_{\bf G}({\bf H})/{\bf H}(k)$
is a continuous homomorphism with respect to which $\phi$ is $T'$-equivariant.
\end{prop}

\begin{proof}
The same as the proof of \cite[Theorem~4.6]{BF-MSR}, see also the proof of Proposition~\ref{prop:Tyonedacomp}.
\end{proof}

Next, in an analogy with \S\ref{sec:trivproper}, 
we allow ourselves to vary $k$, ${\bf G}$ and $\delta$.

\begin{defn} \label{def:trivalg2}
Fix an lcsc group $S$ and a closed subgroup $\Gamma<S$.
Let $T<S$ be a closed subgroup.
We will say that $S$ has only trivial $T$-algebraic representations if 
for any complete separable valued field $k$, a $k$-algebraic group ${\bf G}$ and a continuous homomorphism $\delta:\Gamma\to {\bf G}(k)$,
there exit a $k$-algebraic subgroup ${\bf H}<{\bf G}$ normalized by $\delta(\Gamma)$,
a continuous homomorphism $\theta:S\to N_{\bf G}({\bf H})/{\bf H}(k)$ and $n\in N_{\bf G}({\bf H})(k)$ such that the composition of $\inn(n)\circ\delta:\Gamma\to N_{\bf G}({\bf H})(k)$ with the
quotient $N_{\bf G}({\bf H})(k)\to N_{\bf G}({\bf H})/{\bf H}(k)$ equals $\theta|_\Gamma$
and there exists an a.e defined measurable $\phi:S\to {\bf G}/{\bf H}(k)$ which is left $\Gamma$-equivariant via $\delta$ and right $S$-equivariant via $\theta$
such that $\phi$ together with $\theta|_T$
form an initial object in corresponding category of 
$T$-algebraic representations of $S$.
\end{defn}

\begin{theorem} \label{thm:Tnoalg}
Let $S$ be a locally compact second countable group and $\Gamma<S$ a closed subgroup.
Assume that there exists a topologically generating (finite or infinite) sequence of abelian subgroup $T_0,T_1,\ldots <S$ such that for every $i>0$,
$T_i$ commutes with $T_{i-1}$ and for every $i\geq 0$, $T_i$ acts ergodically on $S/\Gamma$.
Assume also that $S/N$ has trivial abelianization, where $N$ is the closed normal subgroup generated by $\Gamma$.
Let $T=T_0$.
Then $S$ has only trivial $T$-algebraic representations.
\end{theorem}

\begin{proof}
The proof is the essentially the same as the proof of Theorem~\ref{thm:Tnoproper}
and we will not repeat it.
\end{proof}

\section{Proof of Theorem~\ref{thm:main}} \label{sec:main}

Let $\Gamma<S$ be a closed subgroup,
$k$ be a local field and ${\bf G}$ a connected adjoint $k$-simple algebraic group.
Let $G={\bf G}(k)$.
Let $\delta:\Gamma \to G$
be a continuous homomorphism
which is unbounded and Zariski dense in $G$.
We will show that there exists a unique continuous homomorphism $\rho:S\to G$
such that $\delta=\rho|_{\Gamma}$.

We first argue that such an extension, if exists, is unique.
Assume by negation that $\rho_1,\rho_2:S\to G$ are two distinct continuous homomorphism with $\rho_1|_\Gamma=\rho_2|_\Gamma=\delta$.
As $\delta(\Gamma)<G$ is Zariski dense, we get that the Zariski closure of the image of the map $(\rho_1,\rho_2):S\to G\times G$ given by $s\mapsto (\rho_1(s),\rho_2(s))$
strictly contains the diagonal $\Delta{\bf G}<{\bf G}\times {\bf G}$. 
By the maximality of the diagonal (as ${\bf G}$ is $k$-simple) it follows that the image of $(\rho_1,\rho_2)$ is Zariski dense in ${\bf G}\times {\bf G}$.
Composing $(\rho_1,\rho_2)$ with $G\times G\to ({\bf G}\times {\bf G})/\Delta{\bf G}(k)$
we get a map $S\to ({\bf G}\times {\bf G})/\Delta{\bf G}(k)$ which is right $\Gamma$-invariant and $S$-equivariant via $(\rho_1,\rho_2)$.
This map factors via $S/\Gamma$ and we may view the resulting $S$-equivariant map $\psi:S/\Gamma\to ({\bf G}\times {\bf G})/\Delta{\bf G}(k)$
as an algebraic representation of the $S$-space $S/\Gamma$.
However, by Theorem~\ref{thm:noproper}, the $S$-space $S/\Gamma$ has only trivial representation (even as a $T_0$-space),
thus $\psi$ is essentially constant and its essential image is $(\rho_1,\rho_2)(\Gamma)$-invariant.
But the latter group is Zariski dense in ${\bf G}\times {\bf G}$, thus this point is ${\bf G}\times {\bf G}$-invariant.
This is an absurd, as there are no ${\bf G}\times {\bf G}$-invariant points in the homogeneous space $({\bf G}\times {\bf G})/\Delta{\bf G}$.
We conclude that any extension of $\delta$ to $S$, if exists is indeed unique.


We will now show the existence of such an extension.
We will proceed by negation, assuming $\delta$ does not extend to $S$.
We let $T=T_0$.
By Theorem~\ref{thm:Tnoalg} $S$ has only trivial $T$-algebraic representations.
That is, 
there exit a $k$-algebraic subgroup ${\bf H}<{\bf G}$ normalized by $\delta(\Gamma)$,
a continuous homomorphism $\theta:S\to N_{\bf G}({\bf H})/{\bf H}(k)$ and $n\in N_{\bf G}({\bf H})(k)$ such that the composition of $\inn(n)\circ\delta:\Gamma\to N_{\bf G}({\bf H})(k)$ with the
quotient $N_{\bf G}({\bf H})(k)\to N_{\bf G}({\bf H})/{\bf H}(k)$ equals $\theta|_\Gamma$
and there exists an a.e defined measurable $\phi:S\to {\bf G}/{\bf H}(k)$ which is left $\Gamma$-equivariant via $\delta$ and right $S$-equivariant via $\theta$
such that $\phi$ together with $\theta|_T$
form an initial object in corresponding category of 
$T$-algebraic representations of $S$.

As ${\bf H}$ is normalized by the Zariski dense subgroup $\delta(\Gamma)$, we conclude that ${\bf H}$ is normal in the $k$-simple group ${\bf G}$.
Therefore, either ${\bf H}=\{e\}$ or ${\bf H}={\bf G}$.
In the first case $\delta$ extends to $S$ by $\inn(n)^{-1}\circ\theta$, contradicting our assumption.
We therefor have ${\bf H}={\bf G}$.
It follows that every $T$-algebraic representation is essentially constant and its essential value is ${\bf G}$-invariant.
In particular, by taking $\theta$ to be the trivial homomorphism to the trivial group,
We get that every algebraic representation of the $\Gamma$-space $S/T$ is essentially constant and its essential value is ${\bf G}$-invariant.

Note that the action of $\Gamma$ on $S/T$ is amenable, as $T$ is commutative hence amenable.
Setting $Y=S/T$ and $R=\Gamma$ in \cite[Theorem~6.1]{BDL} we conclude having a proper representation of the $\Gamma$-space $S/T$,
corresponding to $\delta:\Gamma\to G$.
Taking again $\theta$ to be the trivial homomorphism to the trivial group, we view this representation as a $T$-proper representation of $S$.
By Theorem~\ref{thm:Tnoproper} we know that $S$ has only trivial $T$-proper representations, that is 
there exit a compact subgroup $K<G$ normalized by $\delta(\Gamma)$,
a continuous homomorphism $\theta:S\to N_G(K)/K$ and $n\in N_G(K)$ such that the composition of $\inn(n)\circ\delta:\Gamma\to N_G(K)$ with the
quotient $N_G(K)\to N_G(K)/K$ equals $\theta|_\Gamma$
and there exists an a.e defined measurable $\phi:S\to G/K$ which is left $\Gamma$-equivariant via $\delta$ and right $S$-equivariant via $\theta$
such that $\phi$ together with $\theta|_T$
form an initial object in corresponding category of 
$T$-proper representations of $S$.

We let ${\bf H}$ be the Zariski closure of $K$ in ${\bf G}$.
The composition of $\phi$ with $G\to {\bf G}/{\bf H}(k)$ could be seen as a $T$-algebraic representation of $S$,
and from the fact that 
every $T$-algebraic representation is essentially constant and its essential value is ${\bf G}$-invariant
we conclude that ${\bf H}={\bf G}$, that is $K$ is Zariski dense in ${\bf G}$.
By Lemma~\ref{lem:NL} below we have that $N_G(K)$ is compact.
But $\delta(\Gamma)$ normalizes $K$ and it is assumed unbounded.
This gives the desired contradiction, and the proof follows.
We are only left to verify the following.

\begin{lemma} \label{lem:NL}
Let $k$ be a local filed.
Let ${\bf G}$ be a $k$-algebraic group and denote $G={\bf G}(k)$.
Let $K<G$ be a compact subgroup.
Then $N_G(K)/Z_G(K)$ is compact.
In particular, if ${\bf G}$ is a semisimple group and $K$ is Zariski dense then $N_G(K)$ is compact.
\end{lemma}

This lemma is obtained from the following more general proposition
by applying it to the conjugation action of ${\bf G}$ on itself.

\begin{prop}
Let $k$ be a local filed.
Let ${\bf G}$ be a $k$-algebraic group and ${\bf V}$ a $k$-${\bf G}$-affine variety.
Denote $G={\bf G}(k)$ and $V={\bf V}(k)$.
Let $C\subset V$ be a compact subset.
Then the group $\Stab_G(C)/\Fix_G(C)$ is compact.
\end{prop}

\begin{proof}
Denote by ${\bf U}$ the Zariski closure of $C$ in ${\bf V}$.
It is defined over $k$ by \cite[Theorem~AG14.4]{borel}.
Without loss of generality, using \cite[Proposition~1.7]{borel}, we may replace ${\bf G}$ by $\Stab_{\bf G}({\bf U})$ and then assume ${\bf V}={\bf U}$.
We then may further assume ${\bf G}=\Stab_{\bf G}({\bf U})/\Fix_{\bf G}({\bf U})$. We do so.
By \cite[Proposition~1.12]{borel} there exists an embedding $V\to k^n$, which we may assume spanning
and 
equivariant with respect to some representation $G\to\GL_n(k)$, which we thus may assume injective.
Denote by $D$ the image of $C$ in $V$ and by $K$ the image of $\Stab_G(C )$ in $\GL_n(k)$.
Then $K$ preserves the balanced convex hull of $D$ given by
\[ E=\left\{\sum_{i=1}^n \alpha_iv_i~|~v_i\in D,~\alpha_i\in k,~\sum_{i=1}^n |\alpha_i|\leq 1\right\}, \]
and the associated Minkowski norm on $k^n$ defined by
\[ \|x\|=\sup\{r>0~|~\forall \alpha\in k,~|\alpha|<r~\Rightarrow~\alpha x\in E\}. \]
Thus $K$ is compact.
\end{proof}



\begin{bibdiv}
\begin{biblist}



\bib{BCL}{article}{
   author = {{Bader}, U.},
author = {{Caprace}, P.-E.},
author = {{L{\'e}cureux}, J.},
    title = {On the linearity of lattices in affine buildings and ergodicity of the singular Cartan flow},
  journal = {ArXiv e-prints},
   eprint = {1608.06265},
     year = {2016},
   adsurl = {http://adsabs.harvard.edu/abs/2016arXiv160806265B},
  adsnote = {Provided by the SAO/NASA Astrophysics Data System}
}





\bib{BDL}{article}{
AUTHOR = {Uri Bader}, Author={Bruno Duchesne}, Author={Jean Lecureux},
     TITLE = {Almost algebraic actions of algebraic groups and applications
              to algebraic representations},
   JOURNAL = {Groups Geom. Dyn.},
  FJOURNAL = {Groups, Geometry, and Dynamics},
    VOLUME = {11},
      YEAR = {2017},
    NUMBER = {2},
     PAGES = {705--738},
      ISSN = {1661-7207},
   MRCLASS = {20G15 (14L30 37C40)},
  MRNUMBER = {3668057},
       DOI = {10.4171/GGD/413},
       URL = {https://doi.org/10.4171/GGD/413},
}


\bib{AREA}{article}{
   author = {Bader, Uri}, author={Furman, Alex},
    title = {Algebraic Representations of Ergodic Actions and Super-Rigidity},
  journal = {ArXiv e-prints},
   eprint = {1311.3696},
     year = {2013},
   adsurl = {http://adsabs.harvard.edu/abs/2013arXiv1311.3696B},
  adsnote = {Provided by the SAO/NASA Astrophysics Data System}
}





\bib{BF-products}{article}{
AUTHOR = {Bader, Uri}, Author={Furman, Alex},
    title = {Super-Rigidity and non-linearity for lattices in products},
  journal = {ArXiv e-prints},
   eprint = {1802.09931},
     year = {2018},
   adsurl = {http://adsabs.harvard.edu/abs/2018arXiv180209931B},
  adsnote = {Provided by the SAO/NASA Astrophysics Data System}
}

\bib{BF-MSR}{article}{
AUTHOR = {Bader, Uri}, Author={Furman, Alex},
    title = {An extension of Margulis' Super-Rigidity Theorem},
  journal = {not yet},
     year = {2018},
}




\bib{Bader-Gelander}{article}{
    AUTHOR = {Bader, Uri},
    AUTHOR = {Gelander, Tsachik},
     TITLE = {Equicontinuous actions of semisimple groups},
   JOURNAL = {Groups Geom. Dyn.},
  FJOURNAL = {Groups, Geometry, and Dynamics},
    VOLUME = {11},
      YEAR = {2017},
    NUMBER = {3},
     PAGES = {1003--1039},
      ISSN = {1661-7207},
   MRCLASS = {22E46 (22D40)},
  MRNUMBER = {3692904},
MRREVIEWER = {Luciana A. Alves},
       DOI = {10.4171/GGD/420},
       URL = {https://doi.org/10.4171/GGD/420},
}

\bib{borel}{book}{
    AUTHOR = {Borel, Armand},
     TITLE = {Linear algebraic groups},
    SERIES = {Graduate Texts in Mathematics},
    VOLUME = {126},
   EDITION = {Second},
 PUBLISHER = {Springer-Verlag},
   ADDRESS = {New York},
      YEAR = {1991},
     PAGES = {xii+288},
      ISBN = {0-387-97370-2},
   MRCLASS = {20-01 (20Gxx)},
  MRNUMBER = {1102012 (92d:20001)},
MRREVIEWER = {F. D. Veldkamp},
       DOI = {10.1007/978-1-4612-0941-6},
       URL = {http://dx.doi.org/10.1007/978-1-4612-0941-6},
}


\bib{valuedfields}{article}{
    AUTHOR = {Bosch, S.}, Author={G\"untzer, U.}, author={Remmert, R.},
     TITLE = {Non-{A}rchimedean analysis},
    SERIES = {Grundlehren der Mathematischen Wissenschaften [Fundamental
              Principles of Mathematical Sciences]},
    VOLUME = {261},
      NOTE = {A systematic approach to rigid analytic geometry},
 PUBLISHER = {Springer-Verlag, Berlin},
      YEAR = {1984},
     PAGES = {xii+436},
      ISBN = {3-540-12546-9},
   MRCLASS = {32K10 (30G05 46P05)},
  MRNUMBER = {746961},
MRREVIEWER = {W. Bartenwerfer},
       DOI = {10.1007/978-3-642-52229-1},
       URL = {https://doi.org/10.1007/978-3-642-52229-1},
}









\bib{effros}{article}{
    AUTHOR = {Effros, Edward G.},
     TITLE = {Transformation groups and {$C^{\ast} $}-algebras},
   JOURNAL = {Ann. of Math. (2)},
  FJOURNAL = {Annals of Mathematics. Second Series},
    VOLUME = {81},
      YEAR = {1965},
     PAGES = {38--55},
      ISSN = {0003-486X},
   MRCLASS = {46.65},
  MRNUMBER = {0174987 (30 \#5175)},
MRREVIEWER = {J. M. G. Fell},
}







\bib{margulis-book}{book}{
    AUTHOR = {Margulis, G. A.},
     TITLE = {Discrete subgroups of semisimple {L}ie groups},
    SERIES = {Ergebnisse der Mathematik und ihrer Grenzgebiete (3) [Results
              in Mathematics and Related Areas (3)]},
    VOLUME = {17},
 PUBLISHER = {Springer-Verlag},
   ADDRESS = {Berlin},
      YEAR = {1991},
     PAGES = {x+388},
      ISBN = {3-540-12179-X},
   MRCLASS = {22E40 (20Hxx 22-02 22D40)},
  MRNUMBER = {1090825 (92h:22021)},
MRREVIEWER = {Gopal Prasad},
}

\bib{Rosendal}{article}{
      author={Rosendal, Christian},
       title={Automatic continuity of group homomorphisms},
        date={2009},
     journal={Bull. Symbolic Logic},
      volume={15},
      number={2},
       pages={184-- 214},
}




		





\bib{zimmer-book}{book}{
   author={Zimmer, R. J.},
   title={Ergodic theory and semisimple groups},
   series={Monographs in Mathematics},
   volume={81},
   publisher={Birkh\"auser Verlag},
   place={Basel},
   date={1984},
   pages={x+209},
   isbn={3-7643-3184-4},
   review={\MR{776417 (86j:22014)}},
}

\end{biblist}
\end{bibdiv}
\end{document}